\documentclass{amsart}
\usepackage{setspace}
\usepackage{a4}
\usepackage{amsthm}
\usepackage{latexsym}
\usepackage{amsfonts}
\usepackage{graphicx}
\usepackage{textcomp}
\usepackage{cite}
\usepackage{enumerate}
\usepackage{amssymb}
\usepackage{hyperref}
\usepackage{amsmath}
\usepackage{tikz}
\usepackage[mathscr]{euscript}
\usepackage{mathtools}
\newtheorem{theorem}{Theorem}[section]

\newtheorem{corollary}[theorem] {Corollary}
\newtheorem{definition}[theorem]{Definition}

\newtheorem{question}[theorem]{Question}
\title{This is the title}
\usepackage{fancyhdr}

\begin{document}
\hrule\hrule\hrule\hrule\hrule
\vspace{0.3cm}	
\begin{center}
{\bf{CONTINUOUS RANKIN BOUND FOR HILBERT AND BANACH SPACES}}\\
\vspace{0.3cm}
\hrule\hrule\hrule\hrule\hrule
\vspace{0.3cm}
\textbf{K. MAHESH KRISHNA}\\
Post Doctoral Fellow \\
Statistics and Mathematics Unit\\
Indian Statistical Institute, Bangalore Centre\\
Karnataka 560 059, India\\
Email:  kmaheshak@gmail.com\\

Date: \today
\end{center}

\hrule\hrule
\vspace{0.5cm}
\textbf{Abstract}:  Let $(\Omega, \mu)$ be a  measure space and $\{\tau_\alpha\}_{\alpha\in \Omega}$ be a 	normalized continuous Bessel  family for a real Hilbert  space $\mathcal{H}$.  If the diagonal $\Delta\coloneqq \{(\alpha, \alpha):\alpha \in \Omega\}$ is measurable in the measure space $\Omega\times \Omega$, then we show that 
\begin{align}\label{CRBA}
	\sup _{\alpha, \beta \in \Omega, \alpha\neq \beta}\langle \tau_\alpha, \tau_\beta\rangle \geq \frac{-(\mu\times\mu)(\Delta)}{(\mu\times\mu)((\Omega\times\Omega)\setminus\Delta)}.
\end{align}
We call Inequality (\ref{CRBA}) as continuous Rankin bound. It improves 76 years old result of Rankin [\textit{Ann. of Math.,  1947}].  It also answers one of the questions asked by K. M. Krishna in the paper [Continuous Welch bounds with applications, \textit{Commun. Korean Math. Soc., 2023}]. We also derive Banach space version of Inequality (\ref{CRBA}).



\hrule

\hrule
\section{Introduction}

In 1947, Rankin derived following result for a collection of unit vectors in $\mathbb{R}^d$. 
\begin{theorem} (\textbf{Rankin Bound}) \label{RB} \cite{TROPP, RANKIN, RANKIN2}
If $\{\tau_j\}_{j=1}^n$ is a collection of unit vectors in $\mathbb{R}^d$, then 
\begin{align}\label{RB1}
	\max _{1\leq j,k \leq n, j\neq k}\langle \tau_j, \tau_k\rangle \geq\frac{-1}{n-1}.
\end{align}
In particular, 
\begin{align}\label{RB2}
	\min _{1\leq j,k \leq n, j\neq k}\|\tau_j-\tau_k\|^2\leq 	\frac{2n}{n-1}.
\end{align}
\end{theorem}
Striking feature of Inequalities (\ref{RB1}) and  (\ref{RB2}) is that they do not depend upon the dimension $d$.  Inequalities (\ref{RB1}) and  (\ref{RB2})  play important roles in the study of  packings of lines (which  motivated to  study  the packings of planes) \cite{DHILLONHEATHSTROHMER, CONWAYHARDINSLOANE}, Kepler conjecture \cite{HALES, SZPIRO},  sphere packings \cite{BERNDTKOHNENONO, ZONG} and the geometry of numbers \cite{CASSELS}. 

After the derivation of continuous Welch bounds in most general form, author of the paper  \cite{KRISHNA} asked what is the version of Rankin bound for collections indexed by measure spaces. We are going to answer this in this paper.

\section{Continuous Rankin bound}
We start by recalling the notion of  continuous frames which are introduced independently by Ali, Antoine and Gazeau \cite{ALIANTOINEGAZEAU} and Kaiser \cite{KAISER}. In the paper,    $\mathcal{H}$ denotes a real Hilbert space (need not be finite dimensional).
\begin{definition}\cite{ALIANTOINEGAZEAU, KAISER, ALIANTOINEGAZEAUBOOK, GABARDOHAN}
	Let 	$(\Omega, \mu)$ be a measure space. A collection   $\{\tau_\alpha\}_{\alpha\in \Omega}$ in 	a  Hilbert  space $\mathcal{H}$ is said to be a \textbf{continuous frame} (or generalized frame) for $\mathcal{H}$ if the following holds.
	\begin{enumerate}[\upshape(i)]
		\item For each $h \in \mathcal{H}$, the map $\Omega \ni \alpha \mapsto \langle h, \tau_\alpha \rangle \in \mathbb{K}$ is measurable.
		\item There are $a,b>0$ such that 
		\begin{align*}
			a\|h\|^2\leq \int\limits_{\Omega}|\langle h, \tau_\alpha \rangle|^2\,d\mu(\alpha)\leq b \|h\|^2, \quad \forall h \in \mathcal{H}.
		\end{align*}
	\end{enumerate}
 If we do not demand the first inequality in (ii), then we say it is a \textbf{continuous Bessel family} for $\mathcal{H}$. A continuous Bessel family $\{\tau_\alpha\}_{\alpha\in \Omega}$ is said to be normalized or unit norm if $\|\tau_\alpha\|=1$, $\forall \alpha \in \Omega$.
\end{definition}
 Given a continuous Bessel family, the analysis operator
\begin{align*}
	\theta_\tau:\mathcal{H} \ni h \mapsto \theta_\tau h \in \mathcal{L}^2(\Omega);  \quad \theta_\tau h:\Omega \ni \alpha \mapsto  \langle h, \tau_\alpha \rangle \in \mathbb{K}
\end{align*}
is a well-defined bounded linear operator. Its adjoint, the synthesis operator is given by 
\begin{align*}
	\theta_\tau^*:\mathcal{L}^2(\Omega)\ni f \mapsto \int\limits_{\Omega}f (\alpha)\tau_\alpha \,d\mu(\alpha)\in \mathcal{H}.
\end{align*}
By combining analysis and synthesis operators, we get the frame operator, defined as 
\begin{align*}
	S_\tau\coloneqq 	\theta_\tau^*	\theta_\tau:\mathcal{H} \ni h \mapsto\int\limits_{\Omega}\langle h, \tau_\alpha \rangle \tau_\alpha \,d\mu(\alpha)\in \mathcal{H}.
\end{align*}
Note that the integrals are weak integrals (Pettis integrals \cite{TALAGRAND}). With  this machinery, we generalizes Theorem \ref{RB}.
\begin{theorem} (\textbf{Continuous Rankin Bound})  \label{CRB}
 Let $(\Omega, \mu)$ be a  measure space and $\{\tau_\alpha\}_{\alpha\in \Omega}$ be a 	normalized continuous Bessel  family for  a real Hilbert space $\mathcal{H}$. If the diagonal $\Delta\coloneqq \{(\alpha, \alpha):\alpha \in \Omega\}$ is measurable in the measure space $\Omega\times \Omega$, then
 \begin{align}\label{CD0}
 	\sup _{\alpha, \beta \in \Omega, \alpha\neq \beta}\langle \tau_\alpha, \tau_\beta\rangle \geq \frac{-(\mu\times\mu)(\Delta)}{(\mu\times\mu)((\Omega\times\Omega)\setminus\Delta)}.	
 \end{align}
 In particular, 
 \begin{align}\label{CD}
 	\inf _{\alpha, \beta \in \Omega, \alpha\neq \beta}\|\tau_\alpha-\tau_\beta\|^2\leq 2 \left(1+\frac{(\mu\times\mu)(\Delta)}{(\mu\times\mu)((\Omega\times\Omega)\setminus\Delta)}	\right).
 \end{align}
\end{theorem}
\begin{proof}
Since $\mu(\Omega)<\infty$ (see lemma 2.3 in \cite{KRISHNA}), $\chi_\Omega \in \mathcal{L}^2(\Omega)$ and 
\begin{align*}
\int\limits_{(\Omega\times\Omega)\setminus\Delta}|\langle \tau_\alpha, \tau_\beta\rangle|\, d(\mu\times\mu)(\alpha,\beta)\leq 	\int\limits_{(\Omega\times\Omega)\setminus\Delta}\| \tau_\alpha\|\|\tau_\beta\|\, d(\mu\times\mu)(\alpha,\beta)=(\mu\times\mu)((\Omega\times\Omega)\setminus\Delta)<\infty.
\end{align*}
Now by using Fubini's theorem, we get

\begin{align*}
	0&\leq \|\theta^*_\tau \chi_\Omega\|^2=\langle \theta^*_\tau \chi_\Omega, \theta^*_\tau \chi_\Omega\rangle =\left\langle \int\limits_{\Omega}\chi_\Omega (\alpha)\tau_\alpha \,d\mu(\alpha), \int\limits_{\Omega} \chi_\Omega (\beta)\tau_\beta \,d\mu(\beta)\right\rangle\\
	&=\left\langle \int\limits_{\Omega}\tau_\alpha \,d\mu(\alpha), \int\limits_{\Omega} \tau_\beta \,d\mu(\beta)\right\rangle = \int\limits_{\Omega} \int\limits_{\Omega}\langle \tau_\alpha, \tau_\beta\rangle\,d\mu(\alpha)\,d\mu(\beta)=\int\limits_{\Omega\times\Omega}\langle \tau_\alpha, \tau_\beta\rangle\, d(\mu\times\mu)(\alpha,\beta)\\
		&=\int\limits_{\Delta}\langle \tau_\alpha, \tau_\beta\rangle\, d(\mu\times\mu)(\alpha,\beta)+\int\limits_{(\Omega\times\Omega)\setminus\Delta}\langle \tau_\alpha, \tau_\beta\rangle\, d(\mu\times\mu)(\alpha,\beta)\\
	&=\int\limits_{\Delta}\langle \tau_\alpha, \tau_\alpha\rangle\, d(\mu\times\mu)(\alpha,\beta)+\int\limits_{(\Omega\times\Omega)\setminus\Delta}\langle \tau_\alpha, \tau_\beta\rangle\, d(\mu\times\mu)(\alpha,\beta)\\
	&=(\mu\times\mu)(\Delta)+\int\limits_{(\Omega\times\Omega)\setminus\Delta}\langle \tau_\alpha, \tau_\beta\rangle\, d(\mu\times\mu)(\alpha,\beta)\\
	&\leq(\mu\times\mu)(\Delta)+\left(\sup _{\alpha, \beta \in \Omega, \alpha\neq \beta}\langle \tau_\alpha, \tau_\beta\rangle \right)(\mu\times\mu)((\Omega\times\Omega)\setminus\Delta).
\end{align*}
Now writing inner product using norm,  we get 
\begin{align*}
\sup _{\alpha, \beta \in \Omega, \alpha\neq \beta}\langle \tau_\alpha, \tau_\beta\rangle&=\sup _{\alpha, \beta \in \Omega, \alpha\neq \beta}\left(\frac{\| \tau_\alpha\|^2+\| \tau_\beta\|^2-\|\tau_\alpha-\tau_\beta\|^2}{2}\right)	=\sup _{\alpha, \beta \in \Omega, \alpha\neq \beta}\left(\frac{2-\|\tau_\alpha-\tau_\beta\|^2}{2}\right)	\\
&=1-\frac{\inf _{\alpha, \beta \in \Omega, \alpha\neq \beta}\|\tau_\alpha-\tau_\beta\|^2}{2}.
\end{align*}
Therefore 
\begin{align*}
1-\frac{\inf _{\alpha, \beta \in \Omega, \alpha\neq \beta}\|\tau_\alpha-\tau_\beta\|^2}{2}\geq 	\frac{-(\mu\times\mu)(\Delta)}{(\mu\times\mu)((\Omega\times\Omega)\setminus\Delta)}
\end{align*}
which gives 
\begin{align*}
\frac{\inf _{\alpha, \beta \in \Omega, \alpha\neq \beta}\|\tau_\alpha-\tau_\beta\|^2}{2}\leq 1+\frac{(\mu\times\mu)(\Delta)}{(\mu\times\mu)((\Omega\times\Omega)\setminus\Delta)}.	
\end{align*}
\end{proof}
\begin{corollary}
	Theorem  \ref{RB}    follows from  Theorem 	\ref{CRB}.
\end{corollary}
\begin{proof}
	Take $\Omega=\{1,\dots,n\} $ and $\mu$ as the counting measure.
\end{proof}
A remarkable feature of Inequality (\ref{CD0}) is that it allows to derive  Inequality (\ref{CD}). We can not do this  by using   first order continuous Welch bound \cite{KRISHNA}.

Given a   measure space  $(\Omega, \mu)$  with measurable diagonal and a 	normalized continuous Bessel  family $\{\tau_\alpha\}_{\alpha\in \Omega}$  for  a real Hilbert space $\mathcal{H}$, we define 
\begin{align*}
\mathcal{M}(\{\tau_\alpha\}_{\alpha\in \Omega})\coloneqq \sup _{\alpha, \beta \in \Omega, \alpha\neq \beta}\langle \tau_\alpha, \tau_\beta\rangle
\end{align*}
and 
\begin{align*}
\mathcal{N}(\{\tau_\alpha\}_{\alpha\in \Omega})\coloneqq	\inf _{\alpha, \beta \in \Omega, \alpha\neq \beta}\|\tau_\alpha-\tau_\beta\|^2.
\end{align*}
Similar to the problem of Grassmannian frames (see \cite{STROHMERHEATH}), we propose following problem.
\begin{question}
\textbf{Given a   measure space  $(\Omega, \mu)$  with measurable diagonal and a  real Hilbert space $\mathcal{H}$, find 		normalized continuous Bessel  family $\{\tau_\alpha\}_{\alpha\in \Omega}$  for   $\mathcal{H}$, such that 
\begin{align}\label{CRP}
\mathcal{M}(\{\tau_\alpha\}_{\alpha\in \Omega})	=\inf\left\{\mathcal{M}(\{\omega_\alpha\}_{\alpha\in \Omega}) : \{\omega_\alpha\}_{\alpha\in \Omega} \text{ is a normalized continuous Bessel  family for } \mathcal{H}\right\}. 
\end{align}
Equivalently, find 	normalized continuous Bessel  family $\{\tau_\alpha\}_{\alpha\in \Omega}$  for   $\mathcal{H}$, such that 
\begin{align*}
\mathcal{N}(\{\tau_\alpha\}_{\alpha\in \Omega})	=\sup\left\{\mathcal{N}(\{\omega_\alpha\}_{\alpha\in \Omega}) : \{\omega_\alpha\}_{\alpha\in \Omega} \text{ is a normalized continuous Bessel  family for } \mathcal{H}\right\}. 	
\end{align*}
Further, for which measure spaces  $(\Omega, \mu)$ and real Hilbert spaces $\mathcal{H}$, solution to (\ref{CRP}) exists?}
\end{question}

\section{Continuous Rankin bound for Banach spaces}
In this section, we derive continuous Rankin bound for Banach spaces. First we need a notion.
\begin{definition}\cite{KRISHNA2}
	Let $(\Omega, \mu)	$ be a measure space  and $p\in[1, \infty)$. Let $\{\tau_\alpha\}_{\alpha\in \Omega}$ be a collection  in a Banach space  $\mathcal{X}$ and 	$\{f_\alpha\}_{\alpha\in \Omega}$ be a collection in  $\mathcal{X}^*$. The pair $ (\{f_\alpha \}_{\alpha \in \Omega}, \{\tau_\alpha \}_{\alpha \in \Omega}) $ is said to be a \textbf{continuous p-Bessel family} for $\mathcal{X}$ if the following conditions are satisfied.
	\begin{enumerate}[\upshape(i)]
		\item For each $x \in \mathcal{X}$, the map $\Omega \ni \alpha \mapsto  f_\alpha(x)\in \mathbb{K}$ is measurable. 
		\item For each $u \in \mathcal{L}^p(\Omega, \mu)$, the map $\Omega \ni \alpha \mapsto  u(\alpha)\tau_\alpha \in \mathcal{X}$ is measurable. 
		\item The map  (\textbf{continuous analysis operator})
		\begin{align*}
			\theta_f: \mathcal{X}\ni x \mapsto \theta_f  \in \mathcal{L}^p(\Omega, \mu); \quad \theta_f x: \Omega \ni \alpha \mapsto (\theta_f x)(\alpha)\coloneqq f_\alpha(x)\in \mathbb{K} 
		\end{align*}
		is a well-defined bounded linear  operator.
		\item The map  (\textbf{continuous synthesis operator})
		\begin{align*}
			\theta_\tau : \mathcal{L}^p(\Omega, \mu)\ni u \mapsto \theta_\tau u\coloneqq \int\limits_\Omega u(\alpha)\tau_\alpha \,d\mu(\alpha)\in \mathcal{X}	
		\end{align*}
		is a well-defined bounded linear  operator.
	\end{enumerate}
\end{definition}
\begin{theorem}(\textbf{Functional Continuous Rankin Bound})
Let $(\Omega, \mu)$ be a  finite measure space and $(\{f_\alpha\}_{\alpha\in \Omega}, \{\tau_\alpha\}_{\alpha\in \Omega})$ be a 	 continuous p-approximate Bessel family for  a real Banach  space $\mathcal{X}$ satisfying the following. 
\begin{enumerate}[\upshape(i)]
	\item $f_\alpha(\tau_\alpha)=1$ for all $\alpha \in \Omega$.
	\item  $\|f_\alpha\|\leq 1$, $\|\tau_\alpha\|\leq 1$ for all $1\leq \alpha \in \Omega$.
	\item $\theta_f \theta_\tau \chi_\Omega \geq 0$.
\end{enumerate}
 If the diagonal $\Delta\coloneqq \{(\alpha, \alpha):\alpha \in \Omega\}$ is measurable in the measure space $\Omega\times \Omega$, then
\begin{align*}
	\sup _{\alpha, \beta \in \Omega, \alpha\neq \beta}f_\alpha(\tau_\beta) \geq \frac{-(\mu\times\mu)(\Delta)}{(\mu\times\mu)((\Omega\times\Omega)\setminus\Delta)}.	
\end{align*}
\end{theorem}
\begin{proof}
Since $\mu(\Omega)<\infty$, we have
\begin{align*}
	\int\limits_{(\Omega\times\Omega)\setminus\Delta}|f_\alpha(\tau_\beta)|\, d(\mu\times\mu)(\alpha,\beta)\leq 	\int\limits_{(\Omega\times\Omega)\setminus\Delta}\| f_\alpha\|\|\tau_\beta\|\, d(\mu\times\mu)(\alpha,\beta)\leq (\mu\times\mu)((\Omega\times\Omega)\setminus\Delta)<\infty.
\end{align*}	
Now by using Fubini's theorem, we get 	
\begin{align*}
0&\leq \int\limits_{\Omega}(\theta_f \theta_\tau \chi_\Omega)(\alpha)\,d\mu(\alpha)	= \int\limits_{\Omega}f_\alpha (\theta_\tau \chi_\Omega)\,d\mu(\alpha)=\int\limits_{\Omega}f_\alpha \left(\int\limits_{\Omega}\chi_\Omega(\beta)\tau_\beta \,d\mu(\alpha)\right)\,d\mu(\alpha)\\
&=\int\limits_{\Omega}f_\alpha \left(\int\limits_{\Omega}\tau_\beta \,d\mu(\beta)\right)\,d\mu(\alpha)=\int\limits_{\Omega} \int\limits_{\Omega}f_\alpha(\tau_\beta) \,d\mu(\beta)\,d\mu(\alpha)=\int\limits_{\Omega\times\Omega}f_\alpha(\tau_\beta)\, d(\mu\times\mu)(\alpha,\beta)\\
&=\int\limits_{\Delta}f_\alpha(\tau_\beta)\, d(\mu\times\mu)(\alpha,\beta)+\int\limits_{(\Omega\times\Omega)\setminus\Delta}f_\alpha(\tau_\beta)\, d(\mu\times\mu)(\alpha,\beta)\\
&=\int\limits_{\Delta}f_\alpha(\tau_\alpha)\, d(\mu\times\mu)(\alpha,\beta)+\int\limits_{(\Omega\times\Omega)\setminus\Delta}f_\alpha(\tau_\beta)\, d(\mu\times\mu)(\alpha,\beta)\\
&=(\mu\times\mu)(\Delta)+\int\limits_{(\Omega\times\Omega)\setminus\Delta}f_\alpha(\tau_\beta)\, d(\mu\times\mu)(\alpha,\beta)\\
&\leq(\mu\times\mu)(\Delta)+\left(\sup _{\alpha, \beta \in \Omega, \alpha\neq \beta}f_\alpha(\tau_\beta) \right)(\mu\times\mu)((\Omega\times\Omega)\setminus\Delta).
\end{align*}
\end{proof}
\begin{corollary}
Let $\{\tau_j\}_{j=1}^n$ be  a collection in a real Banach space $\mathcal{X}$	and $\{f_j\}_{j=1}^n$ be  a collection in $\mathcal{X}^*$ satisfying the following. 
\begin{enumerate}[\upshape(i)]
	\item $f_j(\tau_j)=1$ for all $1 \leq j \leq n$.
	\item  $\|f_j\|\leq 1$, $\|\tau_j\|\leq 1$ for all $1 \leq j \leq n$.
	\item $\sum_{1\leq j, k \leq n}f_j(\tau_k)\geq 0$. 
\end{enumerate}
Then 
\begin{align*}
		\max _{1\leq j,k \leq n, j\neq k}f_j(\tau_k) \geq\frac{-1}{n-1}.
\end{align*}
\end{corollary}

 \bibliographystyle{plain}
 \bibliography{reference.bib}

\end{document}